\documentclass[12pt,reqno]{amsart}

\usepackage[text={420pt,660pt},centering]{geometry}
\usepackage{amssymb,amsfonts,amsthm}
\usepackage{comment}
\usepackage{enumitem}
\usepackage{bm}

\usepackage[dvipsnames]{color}

\usepackage[colorlinks=true, pdfstartview=FitV, linkcolor=black, citecolor=blue, urlcolor=blue]{hyperref}

\def\Xint#1{\mathchoice
{\XXint\displaystyle\textstyle{#1}}%
{\XXint\textstyle\scriptstyle{#1}}%
{\XXint\scriptstyle\scriptscriptstyle{#1}}%
{\XXint\scriptscriptstyle%
\scriptscriptstyle{#1}}%
\!\int}
\def\XXint#1#2#3{{\setbox0=\hbox{$#1{#2#3}{%
\int}$ }
\vcenter{\hbox{$#2#3$ }}\kern-.6\wd0}}
\def\barint{\, \Xint -} 
\def\bariint{\barint_{} \kern-.4em \barint}
\def\bariiint{\bariint_{} \kern-.4em \barint}
\renewcommand{\iint}{\int_{}\kern-.34em \int} 
\renewcommand{\iiint}{\iint_{}\kern-.34em \int} 

\DeclareMathAlphabet{\mathcal}{OMS}{cmsy}{m}{n}

\theoremstyle{plain}

\newtheorem{theorem}{Theorem}[section]

\newtheorem{lemma}[theorem]{Lemma}

\newtheorem{proposition}[theorem]{Proposition}

\theoremstyle{definition}
\newtheorem{remark}[theorem]{Remark}

\newcommand{\R}{\mathbb{R}}
\newcommand{\C}{\mathbb{C}}
\newcommand{\N}{\mathbb{N}}

\newcommand{\data}{{\rm data}}

\newcommand{\p}{\partial}

\newcommand{\les}{\lesssim}

\newcommand{\norm}[1]{\lVert #1 \rVert}
\renewcommand{\:}{\colon}

\renewcommand{\ss}{{\rm ss}}

\let\div\relax
\DeclareMathOperator{\div}{div}
\DeclareMathOperator{\curl}{curl}

\renewcommand{\L}{\bm{L}}
\newcommand{\BS}{{\rm BS}}
\let\Re\relax
\DeclareMathOperator{\Re}{Re}

\numberwithin{equation}{section}

\title[Non-uniqueness of Leray solutions with hypodissipation]{Non-uniqueness of Leray solutions to the hypodissipative Navier-Stokes equations in two~dimensions}
\author[Albritton]{Dallas Albritton}
\address{School of Mathematics, Institute for Advanced Study, 1 Einstein Dr., Princeton, NJ 08540, USA}
\email{dallas.albritton@ias.edu}

\author[Colombo]{Maria Colombo}
\address{EPFL SB, Station 8, 
CH-1015 Lausanne, Switzerland
}
\email{maria.colombo@epfl.ch}

\date{\today}

\begin{document}

\begin{abstract}
We exhibit non-unique Leray solutions of the forced Navier-Stokes equations with hypodissipation in two dimensions.
Unlike the solutions constructed in~\cite{albritton2021non}, the solutions we construct live at a supercritical scaling, in which the hypodissipation formally becomes negligible as $t \to 0^+$. In this scaling, it is possible to perturb the Euler non-uniqueness scenario of Vishik~\cite{Vishik1,Vishik2} to the hypodissipative setting at the nonlinear level.  Our perturbation argument is quasilinear in spirit and circumvents the spectral theoretic approach to incorporating the dissipation in~\cite{albritton2021non}.
\end{abstract}



\maketitle
\tableofcontents

\parskip   2pt plus 0.5pt minus 0.5pt

\section{Introduction}

We consider the two-dimensional fractional Navier-Stokes equations 
\begin{equation}\label{eq:NS}
\tag{NS$_\beta$}
\left\lbrace
\begin{aligned}
\partial_t u + (u \cdot \nabla) u + \nabla p+\Lambda^\beta u &= f  \\
\div u &= 0
\end{aligned}
\right.
\quad \text{ in } \R^2 \times \R_+
\end{equation}
where $\beta \in (0,2]$ and $\Lambda^\beta =  (-\Delta)^{\beta/2}$ is the fractional Laplacian. 
We consider the Cauchy problem for~\eqref{eq:NS} with divergence-free initial velocity
\begin{equation}\label{e:Cauchy}
u (\cdot, 0) = u_0 \in L^2(\R^2).
\end{equation}

It is not difficult to adapt the seminal work~\cite{leray} of Leray to demonstrate the existence of a global-in-time \emph{Leray-Hopf solution} for each divergence-free initial velocity $u_0 \in L^2(\R^2)$ and body force $f \in L^1_t L^2_x(\R^2 \times \R_+)$. These are distributional solutions $u \in L^\infty_t L^2_x \cap L^2_t \dot H^{\frac{\beta}{2}}_x(\R^2 \times \R_+)$ to~\eqref{eq:NS} satisfying  $\| u(\cdot,t) - u_0 \|_{L^2(\R^2)} \to 0$ as $t \to 0^+$ and the \emph{energy inequality} for all $t >0$:
\begin{equation}\label{e:energy_ineq}
\frac{1}{2}\int_{\R^2} |u(x,t)|^2 \, dx + \int_0^t \int_{\R^2} |\Lambda^{\frac{\beta}{2}} u|^2 \, dx\, ds \leq \frac{1}{2}\int_{\R^3} | u_0(x) |^2 \, dx + \int_0^t \int_{\R^2} f \cdot u \, dx \, ds \, .
\end{equation}

It is well known that Leray's solutions are unique in the classical setting $\beta = 2$. In this paper, we demonstrate non-uniqueness when $\beta < 2$:

\begin{theorem}[Non-uniqueness of Leray-Hopf solutions]
\label{thm:main}
For all $\beta \in (0,2)$, there exist a force $f \in L^1_t L^2_x(\R^2 \times \R_+)$
and two distinct Leray-Hopf solutions, $u$ and $\bar{u}$, with zero initial velocity $u_0 \equiv 0$ and force $f$:
\begin{equation}
	u \neq \bar{u} \text{ on } \R^2 \times (0,T) \text{ for all } T > 0 \, .
\end{equation}
\end{theorem}

The solutions are constructed explicitly and satisfy many desirable properties, among which is smoothness on $\R^2 \times (0,T)$ for $T \ll 1$ and energy equality in~\eqref{e:energy_ineq}.

In a recent work~\cite{albritton2021non} with Bru{\'e}, we constructed non-unique Leray solutions of the three-dimensional Navier-Stokes equations. While the method therein is also applicable to~\eqref{eq:NS}, at least when $\beta > 1$, our goal is to highlight a different but equally natural way of incorporating viscosity into the scheme.

The main idea of the proof is to perturb the examples of non-uniqueness discovered by Vishik in~\cite{Vishik1,Vishik2} to the viscous setting.
To begin, we recall the basic dimensional analysis for the Euler equations
\begin{equation}
    \label{eq:evelocity}
	\p_t u + u \cdot \nabla u + \nabla p = f \, , \quad \div u = 0 \, ,
\end{equation}
which can be equivalently rewritten in vorticity formulation as
\begin{equation}
    \label{eq:evorticity}
    \p_t \omega + u \cdot \nabla \omega = g \, , \quad u = \Delta^{-1} \nabla^\perp \omega \, ,
\end{equation}
where $\omega =\curl u$ is the vorticity and  $g = \curl f$.
They have a two-parameter scaling symmetry
\begin{equation}
	u_{\mu,\eta} = \frac{\eta}{\mu} u(\mu x, \eta t) \,, \qquad 	\omega_{\mu,\eta} = \eta u(\mu x, \eta t)
\end{equation}
corresponding to the physical dimensions $[x] = L$, $[t] = T$, $[u] = L/T$, $[\omega] = 1/T$, $[f] = L/T^2$, and $[g] = 1/T^2$.
We suppose a relationship $T = L^\alpha$ between the dimensions. This yields a one-parameter scaling symmetry:
\begin{equation}
	u_\lambda = \lambda^{\alpha - 1} u(\lambda x, \lambda^\alpha t) \, , \qquad
	\omega_\lambda = \lambda^\alpha u(\lambda x, \lambda^\alpha t) \, .
\end{equation}
One may seek special solutions $\bar{\omega}$ which are invariant under the above symmetry and necessarily have the form
\begin{equation}
   \label{eq:baromegadef}
	\bar{\omega}(x) = \frac{1}{t} \bar{\Omega} (\xi) \, ,
\end{equation}
where $\xi = x/t^{\frac{1}{\alpha}}$. In fact, we may analyze an arbitrary Euler solution $\omega$ in \emph{similarity variables} $(\xi,\tau)$, where $\tau = \log t$ and the \emph{similarity profile} $\Omega$ is defined by
\begin{equation}
    \label{eq:omegadef}
	\omega(x,t) = \frac{1}{t} \Omega (\xi,\tau) \,, \qquad  g(x,t) = \frac{1}{t^2} G (\xi,\tau) \,.
\end{equation}
The Euler equations~\eqref{eq:evorticity} in similarity variables are
\begin{equation}
    \label{eq:similarityeuler}
	\p_\tau \Omega + (-1- \frac{\xi}{\alpha} \cdot \nabla) \Omega + U \cdot \nabla \Omega = G \, , \quad U = \nabla^\perp \Delta^{-1} \Omega \, .
\end{equation}
Self-similar solutions $\bar{\omega}$ correspond precisely to steady states $\bar{\Omega}$ in the new variables.

If $\bar{\Omega}$ is an unstable steady state of~\eqref{eq:similarityeuler}, then it is natural to seek its \emph{unstable manifold}. Suppose that $\Omega^{\rm lin}$ 
is an exponentially growing solution of the linearization of~\eqref{eq:similarityeuler} around $\bar{\Omega}$. 
One seeks a trajectory
\begin{equation}
    \label{eq:basicansatz}
    \Omega = \bar{\Omega} + \Omega^{\rm lin} + \Omega^{\rm per} \, ,
\end{equation}
where $\Omega^{\rm per} = o_{\tau \to -\infty}(|\Omega^{\rm lin}|)$  ensures that $\Omega$ is not identically equal to $\bar{\Omega}$. On the other hand, $\Omega \overset{\tau \to -\infty}{\longrightarrow} \bar{\Omega}$. This is the non-uniqueness scenario in~\cite{Vishik1}, which demonstrated the sharpness of the Yudovich class~\cite{Yudovich1963} in the `forced category'.



How does the above non-uniqueness scenario look with viscosity? The fractional Navier-Stokes equations~\eqref{eq:NS} in self-similarity variables are
\begin{equation}
    \label{eq:NSss}
	\p_\tau \Omega + (-1-\xi \cdot \nabla_\xi / \alpha) \Omega + U \cdot \nabla \Omega + e^{\tau \gamma} \Lambda^\beta \Omega = G \, , \quad U = \nabla^\perp \Delta^{-1} \Omega \, ,
\end{equation}
where
\begin{equation}
	\gamma = 1-\frac{\beta}{\alpha} \, .
\end{equation}

The traditional choice of exponent is $\alpha = \beta$, so that the PDE~\eqref{eq:NSss} is $\tau$-autonomous. (The choice $\alpha = \beta = 2$ in~\cite{jiasverakselfsim,albritton2021non} is the classical Navier-Stokes scaling.) This choice is determined by the dimensions of the viscosity: $[\nu] = L^\beta/T$. With this choice, it is possible to follow the strategy in~\cite{albritton2021non} to prove Theorem~\ref{thm:main} when $\beta \in (1,2)$.  The same strategy would be valid in dimension three with $\beta \in (1,5/2)$, that is, up to the Lions exponent. 
The most delicate part of the analysis in~\cite{albritton2021non} is an eigenvalue perturbation argument which shows that not only the linearized $2$D Euler equations, but also the linearized $3$D Navier-Stokes equations admit an unstable eigenvalue around a well chosen background solution.

In this paper, we instead consider $\alpha > \beta$, so that $\gamma > 0$ and the dissipation term $e^{\tau \gamma} \Lambda^\beta$ becomes perturbative (formally) as $\tau \to -\infty$. Then we consider the ansatz~\eqref{eq:basicansatz} exactly as in the inviscid setting, in particular, choosing $\Omega^{\rm lin}$ as an unstable mode for $2$D Euler rather than Navier-Stokes, and study $\Omega^{\rm per}$. We take the existence of an unstable inviscid self-similar solution for granted.\footnote{Indeed, the majority of the present work was actually written before we realized that the spectral problem in~\cite{albritton2021non} was manageable. This work is independent from~\cite{albritton2021non} except for the observation that Vishik's vortex can be truncated.}  In this way, we turn our attention away from the spectral problem at the heart of~\cite{albritton2021non} and toward the nonlinear arguments, which become quasilinear in nature and inherently more difficult. The nonlinear scheme will be outlined in Section~\ref{sec:strategy}.

Our strategy can be compared to analogous strategies in the singularity formation literature, which seek \emph{stable manifolds} as $\tau \to +\infty$ of (potentially unstable) steady states in backward self-similarity variables. In contrast, we seek \emph{unstable manifolds} as $\tau \to -\infty$ of unstable steady states in forward self-similarity variables. Notably, the three-dimensional compressible Navier-Stokes `implosion' singularity in~\cite{merle2019implosion} is a perturbation of a compressible Euler singularity; in similarity variables, the viscous term appears with a factor $e^{-\tau \gamma}$. See~\cite{oh2021gradient,chickering2021asymptotically} for perturbative dissipation in the fractional Burgers equation. 

The above comparison highlights the following point: In the absence of forcing, the desired self-similar solutions and  corresponding exponents are \emph{chosen for you}. When the scaling of the self-similar solution is supercritical with respect to the scaling of the viscous equation, the formal calculations give hope to perturb inviscid solutions to viscous solutions. For this reason, we choose to work with general (supercritical) exponents below. 
The present work provides another general method
to incorporate a supercritical viscosity into non-uniqueness scenarios; we hope that this will inform future studies on non-uniqueness scenarios and continuation past singularities of various PDEs in fluid mechanics.


\emph{Previous work}. The important works~\cite{jiasverakselfsim,jiasverakillposed,guillodsverak} proposed a route, accompanied by significant numerical evidence, to the (yet unproven) non-uniqueness of Leray solutions  \emph{without force}. The unstable manifold in similarity variables is already in~\cite[Theorem 4.1]{jiasverakillposed}. A related program for the two-dimensional Euler equations was initiated in~\cite{BressanAposteriori,BressanSelfSimilar}. Non-uniqueness of weak Navier-Stokes solutions was demonstrated via convex integration in~\cite{BuckmasterVicolAnnals}. With hypoviscosity $\beta < 2/3$, convex integration was able to achieve the Leray class ~\cite{mariahypodissipativeonefifth,DeRosa19}. See~\cite{albritton2021non} for a more comprehensive review.

Finally, we recently demonstrated non-uniqueness of Leray solutions to the forced Navier-Stokes solutions in bounded domains via a gluing method~\cite{albritton2022gluing}.

\section{Preliminaries}
\label{sec:preliminaries}

Let $ 0 < \beta < \alpha < 2$. 
Given a compactly supported, smooth velocity field $\bar{U} \in C^\infty_0(B_2;\R^2)$, $\bar{U}(x) = V(r) e_\theta$, with corresponding vorticity $\bar{\Omega} = \curl \bar{U}$, we define the linearized operator around $\bar{\Omega}$,
\begin{equation}
	- \L_{\ss} \Omega = \left( - 1 - \frac{\xi}{\alpha} \cdot \nabla \right) \Omega + \bar{U} \cdot \nabla \Omega + U \cdot \nabla \bar{\Omega} \, ,
\end{equation}
where $U = \nabla^\perp  \Delta^{-1} \Omega$.
 The above operator is considered $\L_\ss \: D(\L_\ss) \subset L^2_m \to L^2_m$, where $L^2_m \subset L^2$ consists of $m$-fold rotationally symmetric $L^2$ functions and $D(\L_\ss) := \{ \Omega \in L^2_m : \L_\ss \Omega \in L^2 \}$.

The following proposition was essentially demonstrated in~\cite{Vishik1,Vishik2,OurLectureNotes,albritton2021non}

\begin{proposition}[Unstable vortex]\label{pro:spectral}
Let $a_0 \geq 0$. There exists a compactly supported, smooth velocity field $\bar{U} \in C^\infty_0(B_2;\R^2)$, $\bar{U}(x) = V(r) e_\theta$, with corresponding vorticity $\bar{\Omega} = \curl \bar{U}$ and satisfying the following properties.

1. \textbf{Linear instability}. There exists an integer $m\geq 2$ such that $\L_{\ss}$ has an unstable eigenvalue. More specifically, there exist $\lambda \in \C$ with $\Re \lambda =: a \geq a_0$ and $\eta\in L^2_m\setminus \{0\}$ satisfying $\L_{\rm ss} \eta = \lambda \eta$.

2. \textbf{Semigroup estimate}.
The operator $\L_{\ss}$ generates a continuous semigroup in $L^2_m$, and for all $\delta > 0$, we have the semigroup estimate
	\begin{equation}\label{eqn:semigr}
	\norm{e^{\tau\L_\ss}}_{L^2_m \to L^2} \les_\delta e^{(a+\delta) \tau} \text{ for all } \tau \geq 0 \, .
	\end{equation}

3. \textbf{Eigenfunction estimates}. Any eigenfunction $\eta$ with unstable eigenvalue $\lambda$ is compactly supported in $B_1$. For any $k \in \N_0$ with $\Re \lambda > k/\alpha - 1$, the above eigenfunction belongs to $H^k$, and its velocity field $\BS[\eta]$ belongs to $H^{k+1}$.
\end{proposition}

\begin{proof}
1. and 2. above are contained in~\cite[Theorem 2.4.2]{OurLectureNotes} except for the compact support property of $\bar{U}$, which was pointed out in~\cite[Proposition 2.2]{albritton2021non}. Therefore, we focus on \textbf{3. Eigenfunction estimates}. We present an alternative to the approach in~\cite[Lemma 5.0.3]{OurLectureNotes}, where it was shown that $\eta \in W^{2,\infty}$. 
Suppose that $\eta \in D(\L_{\rm ss})$ is an eigenfunction with unstable eigenvalue $\lambda = \lambda' - 1$:
\begin{equation}
    \left( \lambda' - \frac{\xi}{\alpha} \cdot \nabla + \bar{U} \cdot \nabla \right) \eta + {\rm BS}[\eta] \cdot \nabla \bar{\Omega} = 0 \, .
\end{equation}

First, we demonstrate that $\eta$ is compactly supported. When $|\xi| \geq 2$ is beyond the support of $\bar{U} \in C^\infty_0(B_2)$, we have
\begin{equation}
    \frac{r}{\alpha} \p_r \eta = \lambda' \eta \, , \quad r \geq 2 \, ,
\end{equation}
whose only decaying solutions are identically zero in $\R^2 \setminus B_2$.

Second, we prove that $\eta \in H^k$ provided that $\Re \lambda' > k/\alpha$. We consider ${\rm BS}[\eta] \cdot \nabla \bar{\Omega}$ to be a `forcing term' in the resolvent problem, and we rewrite the above equation via the Laplace transform characterization of the resolvent: $R(\lambda',\L_\ss) = \int_0^{+\infty} e^{-\lambda' t} e^{t \L_\ss} \, dt$. Let $X_t$ be the flow map associated to the autonomous vector field $V(\xi) := \xi/\alpha - \bar{U}(\xi)$. Then
\begin{equation}
    \eta = - \int_0^{+\infty} e^{-\lambda' t} ({\rm BS}[\eta] \cdot \nabla \bar{\Omega}) \circ X_t \, dt \, .
\end{equation}
To estimate $\eta$, we recall basic properties of $X_t$, which is explicitly
\begin{equation}
    X_t(r_0 e^{i \theta_0}) = r(t,r_0) e^{i\theta(t,\theta_0,r_0)} \, ,
\end{equation}
\begin{equation}
    r(t,r_0) = e^{t/\alpha} r_0 \, , \quad \theta(t,\theta_0,r_0) = \theta_0 - \int_0^t \bar{U}_\theta(r(s,r_0)) \, ds \, .
\end{equation}
By the general formulae
\begin{equation}
    (\nabla X_t e_r)(r_0 e^{i \theta_0}) = \p_{r_0} r \; (e_r \circ X_t) + r \p_{r_0} \theta \;  (e_\theta \circ X_t) \, ,
\end{equation}
\begin{equation}
     (\nabla X_t e_\theta)(r_0 e^{i \theta_0}) = \frac{1}{r_0} \p_{\theta_0} r\;   (e_r \circ X_t) + \frac{r}{r_0} \p_{\theta_0} \theta \;  (e_\theta \circ X_t) \, ,
\end{equation}
it is not difficult to demonstrate that $|\nabla X_t| \les e^{t/\alpha}$. A different way of saying this is that solutions of the vector-valued ODE $\dot{\vec{y}} = (\nabla V)(X_t) \vec{y}$ satisfy $|\vec{y}(t)| \les e^{t/\alpha} |\vec{y}(0)|$. By taking spatial derivatives of the ODE $\dot X_t = V(X_t)$, this is enough to bootstrap the estimate on $\nabla X_t$ to $|\nabla^k X_t| \les_k e^{k t/\alpha}$ for all $k \geq 0$. In particular, we have
\begin{equation}
    \| \eta \|_{H^k} \les \int_0^{+\infty} e^{(k/\alpha - \lambda')t} \, dt \; \| {\rm BS}[\eta] \cdot \nabla \bar{\Omega} \|_{H^k} \les \| \eta \|_{H^{k-1}} \, ,
\end{equation}
where the implied constant depends on various quantities. Hence, one can bootstrap $\eta \in L^2$ to $\eta \in H^k$ provided that $\Re \lambda > k/\alpha - 1$. Finally, since $\eta$ is compactly supported and mean zero, we have $|\BS[\eta]| \les |\xi|^{-2}$ for $|\xi| \geq 2$. In particular, $\BS[\eta] \in L^2$. 
\end{proof}

\begin{remark}
It is natural that the unstable eigenfunctions should have smoothness related to $\Re \lambda$ since solutions to the ODE $\lambda f = r \p_r f$, $\Re \lambda > 0$, are $f = \text{const}. \times r^{\lambda}$. In Proposition~\ref{pro:spectral}, once it is known that the eigenfunctions are $C^\beta$, it is possible to bootstrap them in the $C^{k,\beta}$ scale, but we do not require this here.
\end{remark}

Recall the similarity variables $\xi = x/t^{1/\alpha}$ and $\tau = \log t$ defined above~\eqref{eq:omegadef}. Given a similarity profile $\bar{\Omega}$ as in Proposition~\ref{pro:spectral}, we define the self-similar solution $\bar{\omega}$ according to~\eqref{eq:baromegadef}. Its velocity field is $\bar{u}(x,t) = t^{-1+1/\alpha} \bar{U}(\xi)$. More generally, given a vorticity $\omega$, we extract its similarity profile $\Omega$ according to~\eqref{eq:omegadef}. Its velocity field is $u(x,t) = t^{-1+1/\alpha} U(\xi,\tau)$. We make the convention that, given a lowercase variable representing vorticity or velocity, its similarity profile is represented by the corresponding uppercase variable, and vice versa.

We now state basic properties of the force, which is induced entirely by the time derivative and fractional Laplacian of $\bar{\omega}$.

\begin{lemma}[Finite energy forcing]
    \label{lem:forcinglemma}
Define $h := \p_t \bar{\omega} + \Lambda^\beta \bar{\omega}$ and $f := \Delta^{-1} \nabla^\perp h = \p_t \bar{u} + \Lambda^\beta \bar{u}$. Then the background velocity $\bar{u}$ solves the fractional Navier-Stokes equations with velocity forcing $f$, which belongs to $L^1_t L^2_x(\R^2 \times (0,T))$ for all $T > 0$. Additionally, $f \in L^\infty_t H^k_x(\R^2 \times (\varepsilon,T))$ for all $k \geq 0$ and $0 < \varepsilon < T < +\infty$.
\end{lemma}
This is a simple consequence of scaling.

\subsection{Strategy}
\label{sec:strategy}

We revisit the ansatz~\eqref{eq:basicansatz}:
\begin{equation}
	\Omega = \bar{\Omega} + \Omega^{\rm lin} + \Omega^{\rm per}
\end{equation}
where
\begin{equation}\label{e:Omega_lin}
\Omega^{\rm lin} (\xi, \tau) :=  \Re (e^{\lambda \tau} \eta) \, ,
\end{equation}
$\lambda$ is a maximally unstable eigenvalue of $\L_\ss$, and $\eta$ is an associated non-trivial eigenfunction. Then $\Omega^{\rm lin}$ is a solution to the linearized evolution equation
\begin{equation}\label{e:evolution_of_Omega_lin}
\partial_\tau \Omega^{\rm lin} - \L_{\rm ss} \Omega^{\rm lin} = 0 \, ,
\end{equation}
 and our goal is to solve for the higher-order correction $\Omega^{\rm per}$. 
The PDE satisfied by $\Omega^{\rm per}$ is
\begin{equation}
    \label{eq:pdeforomegaper}
    \begin{aligned}
    &\p_\tau \Omega^{\rm per} - \L_\ss \Omega^{\rm per} + (U^{\rm lin} + U^{\rm per}) \cdot \nabla \Omega^{\rm per} + U^{\rm per} \cdot \nabla \Omega^{\rm lin} + e^{\tau \gamma} \Lambda^{\beta} \Omega^{\rm per} \\
    &\quad = - U^{\rm lin} \cdot \nabla \Omega^{\rm lin} - e^{\tau \gamma} \Lambda^{\beta} \Omega^{\rm lin} \, ,
    \end{aligned}
\end{equation}
where the right-hand side is considered a source term.
The fractional Laplacian $e^{\tau \gamma} \Lambda^\beta$ will be treated perturbatively. Hence, our approach is essentially quasilinear in nature, and it will be necessary to incorporate the transport $(U^{\rm lin} + U^{\rm per}) \cdot \nabla$ into the operator. Define
\begin{equation}
	- \L[W] = - \L_{\rm ss} + (U^{\rm lin} + W) \cdot \nabla + e^{\tau \gamma} \Lambda^\beta \, ,
\end{equation}
the `main operator' in the PDE, and the associated (formal) solution operator $S[W] = S[W](\tau,s)$. We seek solutions satisfying Duhamel's formula:
\begin{equation}\label{eqn:omega-per}
	\Omega^{\rm per}(\cdot,\tau) = - \int_{-\infty}^\tau S[U^{\rm per}](\tau,s) [( U^{\rm lin}+U^{\rm per}) \cdot \nabla \Omega^{\rm lin} +  e^{s \gamma} \Lambda^{\beta} \Omega^{\rm lin}](\cdot,s) \, ds \, .
\end{equation}

It will be necessary to develop the (strong) solution theory for $\Omega^{\rm per}$ and the solution operator $S$ in high regularity spaces, whose corresponding velocity fields are Lipschitz continuous, e.g., $\Omega \in H^s$, $s > d/2$. It will be more convenient to estimate $S$ in a function space $X \subset W^{1,4}$ which will be defined in Section~\ref{sec:linear}.

There are two main difficulties, which are already present in the Euler equations and have been encountered in previous work, e.g.,~\cite{bardosguostrauss}; the difficulties will only be exacerbated by the dissipation $e^{\tau \gamma} \Lambda^\beta$.

First, there is the apparent derivative loss in adding the transport operator $(U^{\rm lin} + U^{\rm per}) \cdot \nabla$ onto $\L_\ss$. On the one hand, it is necessary to use the spectral information on $\L_\ss$ through the semigroup estimate in $L^2$. On the other hand, it is too na{\"i}ve to write Duhamel's formula, containing expressions like $\int_0^\tau e^{(\tau - s) \L_\ss} (U^{\rm lin} \cdot \nabla \Omega^{\rm per}) \, ds$, and hope to close an estimate at the $L^2$ level. Moreover, in our setting, adding the dissipation operator $e^{\tau \gamma} \Lambda^\beta$ further exhibits a loss of $\beta$ derivatives. The key observation for handling the loss is that the compact term $U^{\rm per} \cdot \nabla \bar{\Omega}$, responsible for the creation of the unstable eigenvalues, gains regularity. Hence, the strategy will be to lose derivatives while applying the semigroup estimate in Duhamel's formula and regain them in a bootstrapping procedure based on energy estimates.\footnote{A different approach is to work with the Lagrangian formulation, in which there is no derivative loss, see the construction of unstable manifolds in~\cite{LinZengCPAM2013,LinZengCorrigendum2014}.}

The second difficulty concerns the operator $\bar{U} \cdot \nabla$ in high regularity spaces. One way to see the difficulty is to pass a derivative into the equation: The PDE for $\p_i \Omega^{\rm per}$ contains the term $\p_i \bar{U} \cdot \nabla \Omega^{\rm per}$, which in principle can change the growth rate. In~\cite{bardosguostrauss}, it was observed that it is enough assume that $\Re \lambda$ is greater than the maximal fluid Lyapunov exponent. 
Our approach will be Eulerian rather than Lagrangian. It is an energy method specific to shear flows and vortices (in short, the idea is to apply $\p_\theta$ first), and crucially, it is well behaved even with the term $e^{\tau \gamma} \Lambda^\beta$. In fact, when $\beta > 1$, the term $e^{\tau \gamma} \Lambda^\beta$ will help us in the bootstrapping procedure: Once we have propagated one derivative, the other $\beta-1$ derivatives needed to close the argument are regained by smoothing. Heuristically, a gain of $\beta$ derivatives costs a factor $e^{\tau \gamma}$, which is why the dissipation term cannot be used to control everything. This is seen in the tracking of the dissipation norm $D$ below.

\section{Linear theory}\label{sec:linear}


In the following,  we will assume that
$a_0 = 8/\alpha$ and $\delta\leq \delta_0 := \min\{ \gamma/4, 1/8, a_0/2\}$. We also fix $\bar \Omega$, $\lambda=a+ib$, $\eta$ and $m$ as in Proposition~\ref{pro:spectral}.
 We denote by $C$ constants whose values depend only on
 \begin{equation}
 \data = (\alpha,\beta, \bar \Omega, a, b, \eta, m, \delta) \, .    
 \end{equation}
 The dependence on $\bar \Omega$ and $ \eta$ can be made more precise through their regularity and size of the support, but this is not needed for the rest of our arguments, and we do not pursue it here. The value of $C$ may change from line to line. When we need to recall a specific constant of this type, we denote it with $C_1, C_2$, ...

This section is devoted to estimates on the solution operator $S[W]$ to the PDE
\begin{equation}\label{eqn:lin-pde}
	\p_\tau \Omega + (-1-\xi \cdot \nabla/\alpha) \Omega + \bar{U} \cdot \nabla \Omega + U \cdot \nabla \bar{\Omega} + (U^{\rm lin} + W) \cdot \nabla \Omega + e^{\tau \gamma} \Lambda^\beta \Omega = 0
\end{equation}
with initial data $\Omega(\cdot,\tau_0) = \Omega_0$ and divergence-free $W$ satisfying the property~\eqref{eq:Wrequirement} below. We assume that $\Omega_0$ and $W$ are $m$-fold rotationally symmetric, so that the semigroup estimate in Proposition~\ref{pro:spectral} can be applied. 
The remaining terms in the PDE~\eqref{eq:pdeforomegaper} will be incorporated in Section~\ref{sec:nonlin} by Duhamel's formula. 

Let us fix a parameter $Q \in [1,3/2)$. In view of the precise structure of the terms in our equation, it turns out that a natural space to perform such estimates is given by the norm
\begin{equation}
	\label{eq:Xnorm}
	\| \Omega \|_X = \| \Omega \|_{L^Q \cap L^2} +\| \p_\theta \Omega \|_{L^Q \cap L^4}  + \| \nabla \Omega \|_{L^2 \cap L^4} \, , 
\end{equation} 
which is equivalent to $\| \Omega \|_{L^Q} + \| \partial_\theta \Omega \|_{L^Q \cap L^4} + \| \nabla \Omega \|_{L^2 \cap L^4}$, but the description in~\eqref{eq:Xnorm} is perhaps more indicative of the estimates we perform.
We also keep track of the quantity
\begin{equation}
	\| \Omega \|_{D(\tau,\tau_0)}^2 = \int_{\tau_0}^{\tau} e^{s \gamma} \| \Lambda^{1+\frac{\beta}{2}} \Omega \|_{L^2}^2 \, ds
\end{equation}
related to the dissipation.

The space $X$ has some useful properties listed below. Let us consider $W \in L^2$ such that $\curl W \in X$.
Since $ 
K_2 =  1_{B_1}
K_2 +  1_{B_1^c}
K_2 \in L^{p} + L^{q}$ for every $p<2,q>2$, we have by Young's convolution inequality that
\begin{equation}
\label{eqn:U-infty}
 \|  W \|_{L^\infty} =  \| K_2 \ast \curl W \|_{L^\infty} \leq C ( \| \curl W \|_{L^Q}+  \|  \curl W \|_{L^4}) \leq  C \| \curl W \|_{X} \, ,
\end{equation}
and similarly,
\begin{equation}
\label{eqn:p-theta-W-infty}
\| \p_\theta W \|_{L^\infty} =  \| K_2 \ast \p_\theta \curl W \|_{L^\infty} \leq C ( \| \p_\theta \curl W \|_{L^Q} +  \| \p_\theta \curl W \|_{L^4}) \leq C \| \curl W \|_X \, .
\end{equation}


The main result of this section is

\begin{proposition}[Estimates for $S$]
\label{pro:transportestimate}
Let $W \in L^\infty((-\infty, 0]; L^2(\R^2))$ be such that
\begin{equation}
	\label{eq:Wrequirement}
	\sup_{\tau \in (-\infty,0)} e^{-a \tau} \norm{\curl W}_X \leq 1.
\end{equation}
Then for every $\delta \leq \delta_0$, there exist constants $T_{\rm max}:= T_{\rm max}(\data) \leq 0$ and $\bar{C} := \bar{C}(\data) > 1$  such that for every $\tau_0 \in (-\infty,T_{\rm max}]$ and $\tau \in [\tau_0,T_{\rm max}]$, we have
\begin{equation}
	\label{eq:transportestimate}
	\|  S[W](\tau,\tau_0) \Omega_0 \|_X  + \|  S[W](\cdot,\tau_0) \Omega_0 \|_{D(\tau_0,\tau)} \leq \bar{C} e^{(a + \delta) (\tau-\tau_0)} \| \Omega_0 \|_X.
\end{equation}
\end{proposition}

\begin{proof}[Proof of Proposition~\ref{pro:transportestimate}]
The constants $T_{\rm max} \leq 0$  and $\bar{C} > 1$ 
will be chosen later. 

For $\tau \in [\tau_0,T_{\rm max}]$, we denote $\Omega(\cdot,\tau) = S[W](\tau,\tau_0) \Omega_0$ the solution of \eqref{eqn:lin-pde}. The proof of Proposition \ref{pro:transportestimate} will be established through the following claim.

\emph{ Let $\tau_1 \in [\tau_0, T_{\rm max}]$ satisfying the following bootstrap assumption:
\begin{equation}\label{eqn:bootstrap}
\mbox{    For all $\tau \in [\tau_0,\tau_1]$, the estimate~\eqref{eq:transportestimate} holds.}
\end{equation}
Then for all $\tau \in [\tau_0,\tau_1]$ we have
\begin{equation}
\label{eqn:finalbound}
\| \Omega(\cdot,\tau) \|_{X} + \| \Omega \|_{D(\tau_0,\tau)} \leq  C_{1}(1
+ \bar C ^{1/2}e^{\zeta T_{\rm max}}) \bar C ^{1/2} e^{(a+\delta)(\tau-\tau_0)} \| \Omega_0 \|_X \, ,
\end{equation}
where $ C_{1}:= C_1(\data)>1$ and $\zeta := \zeta(\data) > 0$.}


Once this claim is established, we easily conclude the proof of Proposition~\ref{pro:transportestimate} by the following argument. 
We fix  $ \bar C =  (4 C_{1}
)^{2}$ and
$T_{\rm max}$ in terms of $\data$, but otherwise independent of $\tau_0$, such that $ \bar C^{1/2} e^{\zeta T_{\rm max}}  \leq 1
$.
Next, we consider $\bar \tau_1$ to be the largest value of $\tau_1 \in (\tau_0, T_{\rm max})$ such that~\eqref{eq:transportestimate}  holds. If $\bar \tau_1 =T_{\rm max}$, our main claim is proved. The case $\bar \tau_1 <T_{\rm max}$ is ruled out by the continuity of $\tau \to\| \Omega(\cdot,\tau) \|_{X}$
and by \eqref{eqn:finalbound}, which implies with our choice of $T_{\rm max}$ that the bound \eqref{eq:transportestimate} is not saturated in $\bar \tau_1$, namely,
\begin{equation}
    \| \Omega(\cdot,\bar \tau_1) \|_{X} + \| \Omega \|_{D(\tau_0,\bar \tau_1)} \leq  \frac{\bar C}{2} e^{(a+\delta)(\bar \tau_1-\tau_0)} \| \Omega_0 \|_X \, .
\end{equation}
Hence, the proof of Proposition~\ref{pro:transportestimate} is complete
provided that \eqref{eqn:finalbound} is established.

\bigskip

{\bf Step 1. [Baseline estimate: Improved $L^2$ bound]} 
{\it
Suppose that the bootstrap assumption \eqref{eqn:bootstrap} holds. 
Then, for all $\tau \in [\tau_0,\tau_1]$,
\begin{equation}\label{eqn:baseline-est}
	\| \Omega(\cdot,\tau) \|_{L^2} \leq C (1+  \bar C
	e^{2\zeta T_{\rm max}} e^{(a+\delta)(\tau-\tau_0)}) \| \Omega_0 \|_X.
\end{equation}
}

Indeed, by Duhamel's formula:
\begin{equation}\label{eqn:duham}
	\Omega (\cdot,\tau) = e^{(\tau-\tau_0)L_{\rm ss}} \Omega_0 - \int_{\tau_0}^\tau e^{(\tau-s)L_{\rm ss}} [(W + U^{\rm lin}) \cdot \nabla \Omega
	+
	e^{s \gamma} 
	\Lambda^\beta \Omega](\cdot,s)  \, ds.
\end{equation}
According to the semigroup estimate \eqref{eqn:semigr}, we have\footnote{ Crucially, this is where we use assumptions on the semigroup $e^{\tau\L_\ss}$ generated by the linearized operator. While the above Duhamel formula `loses derivatives' in the sense that we are estimating in $L^2$ according to `higher' quantities, e.g., $\| \nabla \Omega \|_{L^2}$, it is acceptable to lose derivatives at this level, though it will not be acceptable at the `top tier' ($\| \nabla \Omega \|_{L^4}$). This is common in quasi-linear perturbation arguments and can already be seen in standard proofs of local-in-time existence for the Euler equations.}
\begin{equation}\label{eqn:duham-est}
\begin{split}
    	\|\Omega (\cdot,\tau)\|_{L^2} &\leq  \| e^{(\tau-\tau_0)L_{\rm ss}} \Omega_0\|_{L^2} + \int_{\tau_0}^\tau \|e^{(\tau-s)L_{\rm ss}} [(W + U^{\rm lin}) \cdot \nabla \Omega
	-
	e^{s \gamma} 
	\Lambda^\beta \Omega ] \|_{L^2}(\cdot,s) \, ds
	\\&
	\leq  C e^{(a + \delta)(\tau - \tau_0)} \| \Omega_0 \|_{L^2}
	+ \int_{\tau_0}^\tau e^{(a+\delta) (\tau - s)} [\| (W + U^{\rm lin}) \cdot \nabla \Omega (\cdot,s) \|_{L^2}(\cdot,s) \, ds 
\\&\qquad	+
	 \int_{\tau_0}^\tau e^{(a+\delta) (\tau - s)} e^{s \gamma} 
	\|\Lambda^\beta \Omega  \|_{L^2}(\cdot,s) ]\, ds \, .
\end{split}
\end{equation}

Next, we estimate the first integrand in the right-hand side of \eqref{eqn:duham-est}. 
By \eqref{eqn:U-infty}, the smoothness of $\eta$ in Proposition~\ref{pro:spectral}, the assumption on $W$ \eqref{eq:Wrequirement} and the bootstrap assumption \eqref{eqn:bootstrap}, we obtain that
\begin{align}
\| (W + U^{\rm lin}) \cdot \nabla \Omega(\cdot,s) \|_2 &\leq \| W \|_{\infty} \| \nabla \Omega \|_{2}+
  \| U^{\rm lin} \|_\infty \| \nabla \Omega \|_2
 \\&\leq  e^{a s} \| \Omega(\cdot,s) \|_X
 \les e^{2s \zeta}e^{(a + \delta) (s - \tau_0)}  \| \Omega_0 \|_X.
 \label{eqn:100}
\end{align}

Plugging this estimate in the first integral in \eqref{eqn:duham-est}, we can estimate such integral with $ C\bar C e^{2\zeta  T_{\rm max}}  e^{(a+\delta)(\tau-\tau_0)} \| \Omega_0 \|_X$.

Next, we consider the dissipation term. First, we interpolate between $L^2$ and $\dot H^{1+\frac{\beta}{2}}$ and by Young inequality we have
\begin{equation}
\label{eq:thingiusedtopass}
\begin{aligned}
	\Big( \int_{\tau_0}^\tau e^{s \gamma} \| \Lambda^\beta \Omega \|_{L^2}^2 \Big)^{1/2}\, ds &\leq \Big(\int_{\tau_0}^\tau e^{s \gamma} \| \Lambda^{1+\frac{\beta}{2}} \Omega \|_{L^2}^{2 \frac{2-\beta}{2+\beta}} \| \Omega \|_{L^2}^{\frac{4\beta}{2+\beta}} \, ds \Big)^{1/2}\\
	 &\leq C\Big(\int_{\tau_0}^\tau e^{s \gamma} \big(\| \Lambda^{1+\frac{\beta}{2}} \Omega \|_{L^2}^{2}+  \| \Omega \|_{L^2}^{2}\big) \, ds \Big)^{1/2}\\
	 &\leq
	C\| \Omega \|_{D(\tau_0,\tau)} + C\bar{C} \Big(\int_{\tau_0}^\tau e^{2s\gamma} e^{2(a+\delta)(s-\tau_0)} \| \Omega_0 \|_{X}^2\, ds \Big)^{1/2} \\
	&
	\leq  C\bar{C}  e^{(a+\delta)(\tau - \tau_0)} \| \Omega_0 \|_{X},
	\end{aligned}
\end{equation}
where we used the bootstrap assumption both on $\| \Omega \|_{D(\tau_0,\tau)}$ and on $\| \Omega \|_{L^2}$.
 Then,  using the Cauchy-Schwarz in time and~\eqref{eq:thingiusedtopass},
\begin{equation}\label{diss-baseline}
\begin{aligned}
\int_{\tau_0}^\tau e^{(\tau- s)(a+\delta)} e^{s \gamma/2} \| e^{s \gamma/2} \Lambda^{\beta} \Omega(\cdot,s) \|_{L^2}\, ds 
	&\leq C  e^{\tau \gamma/2} (\bar{C} + 1) e^{(a+\delta)(\tau-\tau_0)} \| \Omega_0 \|_X
.
	\end{aligned}
\end{equation}
This completes the proof of \eqref{eqn:baseline-est}.

We notice that from an interpolation inequality, \eqref{eqn:baseline-est} and the bootstrap assumption it follows also an improved $L^4$ bound of the following form:
\begin{equation}\label{eqn:omega-4-impr}
\begin{split}
\| \Omega  \|_{L^4} &\leq C \| \Omega  \|_{L^2}^{1/2} \| \nabla  \Omega  \|_{L^2}^{1/2} 
\\&\leq [C(1+ \bar C e^{2\zeta T_{\rm max}}) e^{(a+\delta)(\tau-\tau_0)} )]^{1/2} \bar C^{1/2}e^{\frac{a+\delta}{2}(\tau-\tau_0)} \| \Omega_0 \|_X
\\&
\leq C(1+ \bar C^{1/2} e^{\zeta T_{\rm max}}) \bar C^{1/2}e^{(a+\delta)(\tau-\tau_0)} \| \Omega_0 \|_X.
\end{split}
\end{equation}
Notice that this estimate is worse than \eqref{eqn:baseline-est} due to the presence of $ \bar C^{1/2}$; however, being its power strictly less than $1$, the estimate is still good enough.
%


\bigskip

{\bf Step 2. [Improved $L^Q$ bound]}
{\it Suppose that the bootstrap assumption \eqref{eqn:bootstrap} holds. 
Then  for all $\tau \in [\tau_0,\tau_1]$
\begin{equation}\label{eqn:ts-Q}
	\| \Omega(\cdot,\tau) \|_{L^Q} \leq C(C_{\delta_0}+ C \bar C e^{2\zeta T_{\rm max}} e^{(a+\delta)(\tau-\tau_0)} )\| \Omega_0 \|_X.
\end{equation}
}
We perform an $L^Q$ energy estimate on the equation for $\Omega$, namely we multiply the equation by $Q  |\Omega|^{Q-2} \Omega$ and integrate in space. We obtain
\begin{equation}\label{eqn:Q-est}
\frac{d}{dt} \int |\Omega|^Q \, d \xi \leq \Big(- Q+\frac 2 \alpha \Big) \int |\Omega|^Q \, d \xi + Q \int |U \cdot \nabla \bar{\Omega}| | \Omega|^{Q-1}\, d \xi  - Q\int  e^{\tau \gamma} \Lambda^\beta \Omega  |\Omega|^{Q-2} \Omega\, d \xi 
\end{equation}
The last term, related to the dissipation, is known to be negative via the C{\'o}rdoba-C{\'o}rdoba inequality (see, for instance,~\cite[Lemma 2.4]{CoCo}). The main term of interest is the `forcing'
$ U \cdot \nabla \bar{\Omega}$, which  we estimate thanks to \eqref{eqn:U-infty} and the fact that $\bar \Omega$ is compactly supported:
\begin{equation}
\| U \cdot \nabla \bar{\Omega}\|_{L^Q}\leq 	\| U\|_{L^\infty}\| \nabla \bar{\Omega}\|_{L^Q} \leq C \| \Omega\|_{X} \, .    
\end{equation}

We conclude, via a Gronwall estimate on the inequality \eqref{eqn:Q-est}, that \eqref{eqn:ts-Q} holds.
Here we take advantage of the fact that $a_0>\frac {2}\alpha$ to see that the term $ \big(- Q+\frac 2 \alpha \big) \int |\Omega|^Q \, d \xi $ can be neglected in this estimate, since it does not change the exponential in time behavior of the solution.



\bigskip

{\bf Step 3. [Improved $\p_\theta\Omega $ bound]}
{\it 
Suppose that the bootstrap assumption \eqref{eqn:bootstrap} holds. 
Then  for all $\tau \in [\tau_0,\tau_1]$ and for $p = Q,2,4$, we have
\begin{equation}\label{eqn:improvedpthetabound}
	\| \p_\theta \Omega(\cdot,\tau) \|_{L^p} \leq C(1+ \bar C^{1/2} e^{2\zeta T_{\rm max}})\bar C^{1/2} e^{(a+\delta)(\tau-\tau_0)}   \| \Omega_0 \|_X
.
\end{equation}
}
We first observe that
\begin{equation}
	\p_\theta \Lambda^\beta = \Lambda^\beta \p_\theta \, ,
\end{equation}
which follows from differentiating the rotation symmetry
$
	{\rm Rot}_\theta \Lambda^\beta = \Lambda^\beta {\rm Rot}_\theta
$ 
in the angle $\theta \in \R$ at $\theta = 0$.  Similarly $\partial_\theta \curl W =  \curl \partial_\theta W$ and notice that $\p_\theta$ also commutes through the Biot-Savart law.

Commuting $\p_\theta$ into the PDE, we have
\begin{equation}
	\p_\tau \p_\theta \Omega + (-1 - \frac{\xi}{\alpha} \cdot \nabla) \p_\theta \Omega + (\bar{U} + W + U^{\rm lin}) \cdot \nabla \p_\theta \Omega    + e^{\tau \gamma} \Lambda^\beta \p_\theta \Omega = F,
\end{equation}
where
\begin{equation}
	-F = \p_\theta U \cdot  \nabla \bar{\Omega} + (\p_\theta W + \p_\theta U^{\rm lin}) \cdot \nabla \Omega.
\end{equation}

As regards the first term in the force, we estimate it in $L^p$ via Calder{\'o}n-Zygmund as follows
$$\|\p_\theta U \cdot  \nabla \bar{\Omega} \|_{L^p} \leq \|\frac 1 {|x|} \p_\theta U  \|_{L^p} \| |x| |\nabla \bar{\Omega}| \|_{L^\infty} \leq C \|\nabla U  \|_{L^p} \leq C \| \Omega  \|_{L^p}
$$
Using the baseline estimate \eqref{eqn:baseline-est} if $p=Q$ or $p=2$ and the improved bound \eqref{eqn:ts-Q} if $p=4$, we get
\begin{equation*}
\begin{split}
\|\p_\theta U \cdot  \nabla \bar{\Omega} \|_{L^4} &\leq C \| \Omega  \|_{L^4} 
\leq C(1+ \bar C e^{\zeta T_{\rm max}}) \bar C^{1/2} e^{(a+\delta)(\tau-\tau_0)} \| \Omega_0 \|_X .
\end{split}
\end{equation*}

By \eqref{eqn:U-infty}, \eqref{eqn:p-theta-W-infty},  \eqref{eq:Wrequirement} and the bootstrap assumption we have, for $p=2,4$, 
  \begin{equation}
\begin{split}
 \| (\p_\theta W  + \p_\theta U^{\rm lin}) \nabla \Omega \|_{L^p} 
 &\leq C  (\| \p_\theta W \|_{L^\infty} + \|  \p_\theta U^{\rm lin} \|_{L^\infty} ) \| \nabla \Omega \|_{L^p} 
	\\&\leq  C \bar C e^{a\tau} e^{(a+\delta) (\tau - \tau_0)} \| \Omega_0 \|_X 
	.
\end{split}
\end{equation}
For $p=Q$ we modify the previous computation as follows
\begin{equation}
\begin{split}
	\| (\p_\theta W  + \p_\theta U^{\rm lin}) \nabla \Omega \|_{L^Q} &\leq C (\| \p_\theta W \|_{L^{\frac{2Q}{2-Q}}} + \|  \p_\theta U^{\rm lin} \|_{L^{\frac{2Q}{2-Q}}} ) \| \nabla \Omega \|_{L^2} 
	\\&\leq  C (\| \p_\theta \curl W \|_{L^{Q}} + \|  \p_\theta U^{\rm lin} \|_{L^{\frac{2Q}{2-Q}}} ) \| \nabla \Omega \|_{L^2} 
	\\&
	\leq C
	  \bar C  e^{a\tau}e^{(a+\delta) (\tau - \tau_0)} \| \Omega_0 \|_X 
.
\end{split}
\end{equation}

We multiply the equation by $p (\p_\theta \Omega)^{p-1}$ if $p=Q,2,4$ and integrate by parts. After rewriting the fractional dissipation term and using Holder 
 inequality to estimate the term which includes $F$, we get
\begin{equation}
 \frac{d}{d\tau} \int |\p_\theta \Omega|^p \, d\xi + \Big(\frac{2}{\alpha}-p) \int | \p_\theta \Omega|^p \, d\xi + c e^{\tau \gamma} \int |\Lambda^{\frac{\beta}{2}} |\p_\theta \Omega|^{p/2} |^2 \, d\xi \leq C \| F \|_{L^p} \|\p_\theta \Omega\|_{L^p}^{p-1}.
\end{equation}
We neglect the last term in the left-hand side, which has the right sign, divide by $\|\p_\theta \Omega\|_{L^p}^{p-1}$ and sum over $p=Q,2,4$.


\bigskip

{\bf Step 4. [Improved $\nabla \Omega$ bound]}
{\it
	Suppose that the bootstrap assumption \eqref{eqn:bootstrap} holds. 
Then  for all $\tau \in [\tau_0,\tau_1]$ and for $p=2,4$, we have
	\begin{equation}\label{ts:grad1}
	\| \nabla \Omega \|_{L^p} \leq C(1+ \bar C^{1/2} e^{2\zeta T_{\rm max}}) e^{(a+\delta)(\tau-\tau_0)}  \bar C^{1/2} \| \Omega_0 \|_X \, ,
\end{equation}
\begin{equation}\label{ts:grad2}
\|  \Omega \|_{D(\tau_0,\tau)}  \leq C(1+ \bar C^{1/2} e^{2\zeta T_{\rm max}}) e^{(a+\delta)(\tau-\tau_0)}  \bar C^{1/2} \| \Omega_0 \|_X \, .
\end{equation}
}
Let $i=1,2$ and consider the equation for $\p_i \Omega$: 
\begin{equation}
\begin{aligned}
	\label{eq:piomegaequation}
	&\p_\tau \p_i \Omega + (-1-1/\alpha - \xi \cdot \nabla/\alpha) \p_i \Omega + (\bar{U}  + U^{\rm lin} + W) \cdot \nabla \p_i \Omega \\
	&\quad + 
	 \p_i \bar{U} \cdot \nabla \Omega + \p_i( U^{\rm lin} + W) \cdot \nabla \Omega - e^{\tau \gamma} \Lambda^\beta \p_i \Omega = F_i,
	\end{aligned}
\end{equation}
where
\begin{equation}\label{eqn:for-grad}
	-F_i = \p_i U \cdot \nabla \bar{\Omega} + U \cdot \nabla \p_i \bar{\Omega} 
	.
\end{equation}

We multiply by $|\nabla \Omega |^{p-2}\p_i \Omega$, sum over $i=1,2$  and integrate by parts. The last term in the first line of \eqref{eq:piomegaequation} vanishes, after this computation, due to the transport structure. We have
\begin{equation}\label{eqn:full-grad}
\begin{aligned}
	&\frac{1}{p} \frac{d}{d\tau} \int |\nabla \Omega|^p \, d\xi + \Big(\frac{2}{\alpha p} -1 -\frac 1 \alpha \Big) \int |\nabla \Omega|^p \, d\xi -\int e^{\tau \gamma} \Lambda^{\beta} \p_i \Omega |\nabla \Omega|^{ p-2} \p_i \Omega \, d\xi  
	\\
	&\quad + \sum_{i=1}^2 \int 
	\p_i \bar{U} \cdot \nabla \Omega \p_i \Omega |\nabla \Omega|^{p-2} \, d\xi +  \p_i ( U^{\rm lin} + W) \cdot \nabla \Omega  \p_i \Omega  |\nabla \Omega|^{p-2} d\xi \\
	&\quad =   \sum_{i=1}^2  \int F_i \p_i \Omega  |\nabla \Omega|^{p-2} d\xi \, .
	\end{aligned}
\end{equation}
We remark that
\begin{equation}\label{eqn:cancellaz}
\p_i \bar{U} \cdot \nabla \Omega = \zeta' (|x|) \frac{x_i}{|x|} \partial_\theta \Omega + \zeta(|x|) \p_i x^\perp \cdot \nabla \Omega
= \zeta' (|x|) \frac{x_i}{|x|} \partial_\theta \Omega - \zeta(|x|)  \nabla^\perp \Omega
\end{equation}
and we exploit the cancellation of the last term in \eqref{eqn:cancellaz} when multiplied by $\nabla \Omega$. 
Therefore \eqref{eqn:full-grad} rewrites as
\begin{equation}\label{eqn:full-grad}
\begin{aligned}
	&\frac{1}{p} \frac{d}{d\tau} \int |\nabla \Omega|^p \, d\xi + \Big(\frac{2}{\alpha p} -1 -\frac 1 \alpha \Big) \int |\nabla \Omega|^p \, d\xi -\int e^{\tau \gamma} \Lambda^{\beta} \p_i \Omega |\nabla \Omega|^{ p-2} \p_i \Omega \, d\xi   \\
	&\quad =  - \sum_{i=1}^2 \int  \Big(
	\zeta' (|x|) \frac{x_i}{|x|} \partial_\theta \Omega +  \p_i ( U^{\rm lin} + W) \cdot \nabla \Omega +F_i \Big) \p_i \Omega  |\nabla \Omega|^{p-2}  d\xi 
	\\& \leq  C \|
	 \zeta' (|x|) \frac{x_i}{|x|} \partial_\theta \Omega +  \p_i ( U^{\rm lin} + W) \cdot \nabla \Omega +F_i\|_{L^p}^p + \frac 1 2 \int |\nabla \Omega|^p \, d\xi \, .
	\end{aligned}
\end{equation}

We estimate each term in the $L^p$ norm in the right-hand side. By  \eqref{eqn:improvedpthetabound} and since $\zeta'$ is bounded, we have
\begin{equation}
\|\zeta' (|x|) \frac{x_i}{|x|}
 \partial_\theta \Omega\|_{L^p}
\leq \|\zeta' (|x|)
\|_{L^\infty}
\| \partial_\theta \Omega\|_{L^p} \leq  C(1+ \bar C^{1/2} e^{2\zeta T_{\rm max}}) e^{(a+\delta)(\tau-\tau_0)}  \bar C^{1/2} \| \Omega_0 \|_X.
\end{equation}
By the explicit expression of $\Omega^{\rm lin}$, the regularity of $\eta$ in Proposition~\ref{pro:spectral}, 
and \eqref{eqn:U-infty}, we have  $\| \nabla U^{\rm lin}  \|_{L^\infty} \les e^{a \tau} $. Using the Gagliardo-Nirenberg interpolation inequality, Calder{\'o}n-Zygmund and \eqref{eq:Wrequirement}, we have
$$ \| \nabla W \|_{L^\infty} \leq C\| \nabla ^2 W\|_{L^4}^\theta \|\nabla W\|_{L^Q}^{1-\theta} \leq C\| \nabla \curl W \|_{L^4}^\theta \|\curl W\|_{L^Q}^{1-\theta} \leq e^{a\tau}$$ 
for a suitable choice of $\theta \in (0,1)$. We estimate
\begin{equation*}
\begin{split}
\|  \p_i ( U^{\rm lin} + W) \cdot \nabla \Omega \|_{L^p} &\leq  \big( \| \nabla U^{\rm lin}  \|_{L^\infty}  +  \| \nabla W \|_{L^\infty}\big) \|\nabla \Omega \|_{L^p} 
\\&\leq e^{a\tau} e^{(a + \delta) (\tau-\tau_0)} 
 \| \Omega_0 \|_X.
\end{split}
\end{equation*}

Finally, as regards the force appearing in \eqref{eqn:for-grad}, we use the baseline estimate \eqref{eqn:baseline-est} and \eqref{eqn:omega-4-impr}. For the first term in the force, also by Calder{\'o}n-Zygmund, we have 
\begin{equation*}
\begin{split}
\|  \p_i U \cdot \nabla \bar{\Omega} \|_{L^p}&\leq \|\nabla \bar \Omega (\cdot, \tau)\|_\infty \|\nabla U (\cdot, \tau)\|_{L^p} \leq C \|\Omega (\cdot, \tau)\|_{L^p}
\\&\leq  C(C_{\delta_0}+ \bar C e^{\zeta T_{\rm max}}) \bar C^{1/2}e^{(a+\delta)(\tau-\tau_0)} \| \Omega_0 \|_X.
\end{split}
\end{equation*}
For the second term in the force, given that 
$D^2 \bar \Omega \in L^1$ since $\Omega$ is smooth and compactly supported, and by the $L^\infty$ bound on $U$ in \eqref{eqn:U-infty} we obtain 
\begin{align}
\| U \cdot \nabla \p_i \bar{\Omega}\|_{L^p}&\leq\|U \|_{L^{\infty}} \|D^2 \bar \Omega (\cdot, \tau)\|_{L^p}
\leq C (\|\Omega\|_{L^Q}+ \|\Omega\|_{L^4})
\\& \leq C(1+ \bar C^{1/2} e^{\zeta T_{\rm max}}) \bar C^{1/2}e^{(a+\delta)(\tau-\tau_0)} \| \Omega_0 \|_X.
\end{align}



By the previous estimates, \eqref{eqn:full-grad} rewrites for $p=2$ as

\begin{equation}\label{eqn:full-grad2}
\begin{aligned}
	&\frac{1}{2} \frac{d}{d\tau} \int |\nabla \Omega|^2 \, d\xi 
	-\int |\nabla \Omega|^2 \, d\xi + \int e^{\tau \gamma} |\Lambda^{\beta/2} \nabla \Omega|^2 \, d\xi  \\
& \leq C[ (1+ \bar C^{1/2} e^{2\zeta T_{\rm max}})\bar C^{1/2} e^{(a+\delta)(\tau-\tau_0)}   \| \Omega_0 \|_X ]^2.
	\end{aligned}
\end{equation}

We now observe that, when applying the formula \eqref{eqn:full-grad} for $p=4$, the term related to the fractional dissipation (last term in first line) is nonnegative. This is a technical computation that can be seen, for instance, through the Caffarelli extension $(\p_i \Omega)^*$ of $\p_i \Omega$ in $\{ (\xi,\xi_3) \in \R^2 \times \R: y \geq 0\}$ and through the characterization of the fractional laplacian in terms of the extension as follows.
By means of the divergence theorem, we compute 
\begin{align*}
-&\sum_{i=1}^2\int_{\R^2}   \Lambda^{\beta} \p_i \Omega(\xi,\tau) |\nabla \Omega|^{ 2}(\xi,\tau) \p_i \Omega(\xi,\tau) \, dx \\
&= -C\sum_{i=1}^2 \lim_{\xi_3 \to 0^+} \int_{\R^2} \xi_3^b \partial_{\xi_3}  (\p_i \Omega)^*(x,\xi_3,t)  |(\nabla \Omega)^*|^{ 2}(\xi,\tau) (\p_i \Omega)^*(\xi,\tau)\, dx \\
&= C \sum_{i=1}^2\int_{\R^3_+} \overline{\div} ( \xi_3^b \overline{\nabla}  (\p_i \Omega)^*(x,\xi_3,t)  |(\nabla \Omega)^*|^{ 2}(\xi,\tau) (\p_i \Omega)^*(\xi,\tau) ) \, dx \, d\xi_3 
 \\
&= 
 C \sum_{i=1}^2\int_{\R^3_+} ( \xi_3^b \overline{\nabla} \frac{ (\p_i \Omega)^*}{2} \cdot \overline{\nabla} [ |(\nabla \Omega)^*|^{ 2}(\xi,\tau) ]) 
+\xi_3^b |\overline{\nabla}  (\p_i \Omega)^*|^2  |(\nabla \Omega)^*|^{ 2} )  \, d\xi \, d\xi_3 \,,
  \\
&= 
 C \int_{\R^3_+} ( \xi_3^b \overline{\nabla} \frac{| ({\nabla} \Omega)^*|^2}{2} \cdot \overline{\nabla} [ |(\nabla \Omega)^*|^{ 2} ]
+\xi_3^b \Big(\sum_{i=1}^2|\overline{\nabla}  (\p_i \Omega)^*|^2\Big)  |(\nabla \Omega)^*|^{ 2} )  \, d\xi \, d\xi_3 \geq 0\,,
\end{align*}
where the constant depends only on $\beta$ and in the last two equalities we omitted to indicate that all integrands are evaluated at $(\xi,\xi_3,\tau)$. Thanks to the previous estimates, \eqref{eqn:full-grad} implies then
\begin{equation}\label{eqn:full-grad2}
\begin{aligned}
	&\frac{1}{4} \frac{d}{d\tau} \int |\nabla \Omega|^4 \, d\xi -
	\Big( 1 +\frac{1}{2\alpha } \Big)\int |\nabla \Omega|^4 \, d\xi 
 \leq C[ (1+ \bar C^{1/2} e^{2\zeta T_{\rm max}})\bar C^{1/2} e^{(a+\delta)(\tau-\tau_0)}   \| \Omega_0 \|_X ]^4.
	\end{aligned}
\end{equation}

Integrating the previous inequality and \eqref{eqn:full-grad2} between $\tau_0$ and $\tau$, and taking advantage of the fact that the second term in the right-hand side does not influence the final estimate via Gronwall, we deduce that \eqref{ts:grad1} holds and, with $p=2$, that also \eqref{ts:grad2} holds. \end{proof}

%
%
%

\section{Nonlinear theory}\label{sec:nonlin}

In this section, we complete the proof of Theorem~\ref{thm:main}.


\begin{proof}[Proof of Theorem~\ref{thm:main}] Let $\bar \Omega, a, b, \eta, m, \delta_0$ be fixed as in the beginning of Section~\ref{sec:linear}. 

\emph{Step 1. Construction of two distinct solutions via a limiting procedure}. For $k \in \N$, consider $t_k := e^{-k}$. We solve the Cauchy problem for the fractional Navier-Stokes equations~\eqref{eq:NS} with smooth `initial' vorticity $u^{(k)}(\cdot,t_k) = \bar{u}(\cdot,t_k) + u^{\rm lin}(\cdot,t_k)$ and smooth forcing term $f$ defined in Lemma~\ref{lem:forcinglemma}.
For each $k$, a short-time solution $u^{(k)}$ on $\R^2 \times (t_k,T_k)$ exists and is unique in $C([t_k,T_k];L^2 \cap W^{1,4})$, among other function spaces, including $X$.\footnote{When $\beta \leq 1$, one must argue existence and uniqueness in a quasilinear fashion, whereas when $\beta \in (1,2)$ it is possible to develop the well-posedness theory in a semilinear fashion.} Furthermore, these solutions are regular enough to demonstrate energy conservation:
\begin{equation}
	\label{eq:energyconservation}
	\frac{1}{2} \int |u^{(k)}|^2(x,t') \, dx + \int_{t}^{t'} \int |\Lambda^{\frac{\beta}{2}} u^{(k)}|^2 \, dx \,ds = \frac{1}{2} \int |u^{(k)}|^2(x,t) \, dx + \int_{t}^{t'} \int f \cdot u^{(k)} \, dx \, ds.
\end{equation}

 Our desired second solution, violating uniqueness, should arise as $k \to +\infty$. For this purpose, it is necessary to continue the solutions up to some time $\bar t$ independent of $k$. In particular, we will demonstrate that there exists a time $\bar{t} > 0$, independent of $k$, such that the solutions $u^{(k)}$ can be continued to $\R^2 \times (t_k,\bar{t})$, and writing $u^{(k)} = \bar{u} + u^{\rm lin}+ u^{\rm per,(k)}$, the associated vorticity in self-similar variables $\Omega^{\rm per,(k)}$ solves \eqref{eqn:omega-per} in $[-k, \bar \tau]$ and satisfies
\begin{equation}
	\label{eq:mainestimateattheendofthepaper}
	\| \Omega^{{\rm per},(k)}(\cdot,\tau) \|_{X} \leq e^{\tau(a + \delta_0)}, \quad \tau \in [-k,\bar{\tau}],
\end{equation}
where $\bar{\tau} = \log \bar{t}$.

Once we have the above estimate, we may employ standard weak compactness arguments to take the limit $k\to \infty$ and obtain a solution $u= \bar{u} + u^{\rm lin}+ u^{\rm per}$ satisfying the estimate~\eqref{eq:mainestimateattheendofthepaper} with $\Omega^{\rm per}$ instead of $\Omega^{{\rm per},(k)}$ on the time interval $(-\infty,\bar{\tau}]$. In particular, we will have
\begin{equation}
	\| \Omega(\cdot,\tau) - \bar{\Omega}(\cdot,\tau) \|_X \geq \| \Omega^{\rm lin}(\cdot,\tau) \|_X - \| \Omega^{{\rm per}}(\cdot,\tau) \|_{X} \geq Ce^{\tau a} - e^{\tau(a + \delta_0)} > 0 \, ,
\end{equation}
when $\tau \ll 0$, which 
implies non-uniqueness. To demonstrate that $u$ is a Leray solution of~\eqref{eq:NS}, we use the energy conservation~\eqref{eq:energyconservation} with initial time $t = t_k$. The first term on the right-hand side of~\eqref{eq:energyconservation} is simply $\| \bar{u}(\cdot,t_k) + u^{\rm lin}(\cdot,t_k) \|^2_{L^2}/2$, which converges to zero as $t_k \to 0^+$. Hence, we have the energy equality
\begin{equation}
	\label{eq:energyconservation2}
	\frac{1}{2} \int |u|^2(x,t') \, dx + \int_{0}^{t} \int |\Lambda^{\frac{\beta}{2}} u|^2 \, dx \,ds = \int_{0}^{t} \int f \cdot u \, dx \, ds \, .
\end{equation}

\emph{Step 2. The estimate~\eqref{eq:mainestimateattheendofthepaper} via Duhamel's formula}. It remains to verify~\eqref{eq:mainestimateattheendofthepaper}. By continuity in $X$, the estimate~\eqref{eq:mainestimateattheendofthepaper} holds on a small neighborhood of $-k$. We define $\bar \tau_k>-k$ to be the largest non-positive time such that~\eqref{eq:mainestimateattheendofthepaper} holds and we claim that there exists $\bar \tau$ independent of $k$ (which will be fixed at the end of the proof) such that $\bar \tau_k \geq \bar \tau$. 
Therefore, in $[-k, \bar \tau_k]$ we are in the position to apply Duhamel's formula and the assumptions of Proposition~\ref{pro:transportestimate} applied with $\delta =\delta_0$ and $W = U^{\rm per}$ satisfying the requirement~\eqref{eq:Wrequirement}. We conclude that for every  $\tau \in[-k, \bar \tau_k]$,
\begin{equation}
	\Omega^{\rm per}(\cdot,\tau) = \int_{-k}^\tau S[U^{\rm per}](\tau,s) [-(U^{\rm per}+ U^{\rm lin}) \cdot \nabla \Omega^{\rm lin} +e^{s \gamma} \Lambda^{\beta} \Omega^{\rm lin}](\cdot,s) \, ds \, ,
\end{equation}
and that
\begin{align}\label{eqn:duhamel-nonlin-est}\nonumber
	\|\Omega^{\rm per}(\cdot,\tau)\|_{X}&\leq \int_{-k}^\tau \|S[U^{\rm per}](\tau,s)[-(U^{\rm per}+ U^{\rm lin}) \cdot \nabla \Omega^{\rm lin} +e^{s \gamma} \Lambda^{\beta} \Omega^{\rm lin}](\cdot,s) \|_{X} \, ds
\\&
\leq \int_{-k}^\tau e^{(a_0+\delta_0)( \tau - s)}\|-(U^{\rm per}+ U^{\rm lin}) \cdot \nabla \Omega^{\rm lin} +e^{s \gamma} \Lambda^{\beta} \Omega^{\rm lin}\|_{X}(\cdot,s) \, ds \, .
\end{align}

We claim that the integrands in the right-hand side satisfy
\begin{equation}\label{eqn:force-claim}
	\| (U^{\rm per}+ U^{\rm lin}) \cdot \nabla \Omega^{\rm lin}(\cdot,s) \|_X + \| e^{s \gamma} \Lambda^{\beta} \Omega^{\rm lin}(\cdot,s) \|_X \les e^{s(a+2\delta_0)} \, .
\end{equation}

Indeed, 
by the eigenfunction properties in Proposition~\ref{pro:spectral} as well as the structure of $\eta$ in the variable $\theta$ and elementary interpolation, we have
\begin{equation}
	\| e^{s \gamma} \Lambda^{\beta} \Omega^{\rm lin}(\cdot,s) \|_X \leq C e^{s (a+\gamma)}(\| \eta\|_{C^3({\rm supp} \eta) }+ \|\partial_\theta \eta\|_{C^2({\rm supp} \eta) }) \leq C e^{s (a+\gamma)}.
\end{equation}
By \eqref{eqn:U-infty} and \eqref{eqn:p-theta-W-infty} applied to $U^{\rm per}$, the bound \eqref{eq:mainestimateattheendofthepaper} which by assumption holds in $[-k, \bar \tau_k]$, and the bounds on $\Omega ^{\rm lin}$ in terms of $e^{a_0 s}$, we estimate all the terms composing $\| (U^{\rm per}+ U^{\rm lin}) \cdot \nabla \Omega^{\rm lin} (\cdot,s)\|_{X}$ as
\begin{equation*}
    \begin{split}
        \| (U^{\rm per}+ U^{\rm lin}) \cdot \nabla \Omega^{\rm lin}\|_{L^p} &\leq \big(\| U^{\rm per}\|_{L^\infty}+\| U^{\rm lin}\|_{L^\infty}\big) \| \Omega^{\rm lin}\|_{L^p}
\\&
\leq 
C \big(\|\Omega^{\rm per}\|_X +\|\Omega^{\rm lin}\|_X \big) \|\Omega^{\rm lin}\|_{C^2({\rm supp} \eta)}
\leq C e^{2a_0s}   
    \end{split}
\end{equation*}
$$\| \big(\partial_\theta U^{\rm per}+\partial_\theta U^{\rm lin}\big) \cdot \nabla  \Omega^{\rm lin}\|_{L^p} \leq \big(\| \partial_\theta U^{\rm per}\|_{L^\infty} \|+\| \partial_\theta U^{\rm lin}\|_{L^\infty}\big) \| \nabla \Omega^{\rm lin}\|_{L^p}
\leq 
C e^{2a_0s}
$$
$$\|  \big(U ^{\rm per}+U ^{\rm lin}\big)\cdot  \partial_\theta \nabla \Omega^{\rm lin}\|_{L^p} \leq \big( \|  U^{\rm per}\|_{L^\infty}+\|  U^{\rm lin}\|_{L^\infty}\big) \|\partial_\theta \nabla \Omega^{\rm lin}\|_{L^p}
\leq C e^{2a_0s} 
$$
for $p=Q,2,4$ and
\begin{equation*}
\begin{split}
\| \big( \nabla &U^{\rm per} + \nabla U^{\rm lin} \big) \cdot  \nabla \Omega^{\rm lin}\|_{L^p} + \| \big( U^{\rm per}+U^{\rm lin}\big) \cdot  \nabla^2 \Omega^{\rm lin}\|_{L^p}
\\&\leq \big(\|  \nabla U^{\rm per}\|_{L^p}+\|  \nabla U^{\rm lin}\|_{L^p}\big) \| \nabla \Omega^{\rm lin}\|_{L^\infty} +\big( \|  U^{\rm per}\|_{L^\infty}+\|  U^{\rm lin}\|_{L^\infty}\big) \| \nabla^2 \Omega^{\rm lin}\|_{L^p}
\\&\leq C e^{2a_0s} 
\end{split}
\end{equation*}
%
for $p=2,4$.

Therefore, since our choice of $\delta_0$ satisfies $\delta_0 \leq \frac 1 4 \min\{\gamma, a_0\}$, the previous inequalities imply the validity of \eqref{eqn:force-claim}. We obtain then from \eqref{eqn:duhamel-nonlin-est} and these estimates that 
$$\|\Omega^{\rm per}(\cdot,\tau)\|_{X} \leq C e^{(a_0+\delta_0)\bar \tau} \int_{-\infty}^\tau e^{\delta_0 s} \, ds \leq C  e^{(a_0+2 \delta_0)\bar \tau}$$
for every $\tau \in [-k, \bar \tau_k]$.
Choosing finally $\bar \tau$ so that $C e^{\delta_0\bar \tau}<1/2$, where $C$ is the constant appearing in the right-hand side of the previous formula, we find that $\bar \tau_k\geq \bar \tau$. \end{proof}









\subsubsection{Acknowledgments} DA was supported by NSF Postdoctoral Fellowship  Grant No.\ 2002023 and Simons Foundation Grant No.\ 816048.
MC was supported by the SNSF Grant 182565.

\bibliographystyle{alpha} 
\bibliography{nonuniquebibliography}

\end{document}